\documentclass[a4paper,british,oneside]{amsart}

\usepackage[utf8]{inputenc}
\usepackage[T1]{fontenc}
\usepackage{lmodern}
\usepackage{microtype}

\usepackage{babel}

\usepackage{amssymb}
\usepackage{prettyref}

\usepackage[svgnames]{xcolor}
\usepackage[colorlinks,unicode=true,citecolor=ForestGreen,pdfusetitle]{hyperref}
\usepackage[ocgcolorlinks]{ocgx2} 

\usepackage{mathtools} 

\newrefformat{sec}{Section~\ref{#1}}
\newrefformat{sub}{\S~\ref{#1}}
\newrefformat{thm}{Theorem~\ref{#1}}
\newrefformat{prop}{Proposition~\ref{#1}}
\newrefformat{cor}{Corollary~\ref{#1}}
\newrefformat{lem}{Lemma~\ref{#1}}
\newrefformat{conj}{Conjecture~\ref{#1}}
\newrefformat{def}{Definition~\ref{#1}}
\newrefformat{nota}{Notation~\ref{#1}}
\newrefformat{rem}{Remark~\ref{#1}}
\newrefformat{exa}{Example~\ref{#1}}
\newrefformat{exas}{Examples~\ref{#1}}
\newrefformat{fact}{Fact~\ref{#1}}
\newrefformat{app}{Appendix~\ref{#1}}
\newrefformat{item}{(\ref{#1})}
\newrefformat{axiom}{Axiom~\ref{#1}}

\newrefformat{main}{Theorem~\ref{#1}}

\theoremstyle{plain}
\newtheorem{thm}{Theorem}

\newtheorem{prop}[thm]{Proposition}
\newtheorem{cor}[thm]{Corollary}
\newtheorem{lem}[thm]{Lemma}
\newtheorem{conj}[thm]{Conjecture}
\newtheorem*{claim*}{Claim}

\theoremstyle{definition}
\newtheorem{defn}[thm]{Definition}
\newtheorem*{defn*}{Definition}

\newtheorem{fact}[thm]{Fact}
\newtheorem*{que*}{Question}

\theoremstyle{remark}
\newtheorem{rem}[thm]{Remark}
\newtheorem{exa}[thm]{Example}

\numberwithin{thm}{section}
\numberwithin{table}{section}

\DeclareMathOperator{\BigP}{Big}
\DeclareMathOperator{\Res}{Res}
\DeclareMathOperator{\QQ}{Q}
\DeclareMathOperator{\cl}{cl}
\DeclareMathOperator{\crit}{crit}

\DeclareMathOperator{\ot}{ot}

\DeclareMathOperator{\RV}{RV}
\DeclareMathOperator{\rv}{rv}
\DeclareMathOperator{\supp}{supp}

\newcommand{\Nb}{\mathbb{N}}

\newcommand{\Rb}{\mathbb{R}}

\newcommand{\on}{\mathbf{On}}

\newcommand{\KRR}{K((\Rb^{\leq 0}))}

\title{Irreducibility in generalized power series}
\date{27th November 2023}

\author{Antongiulio Fornasiero}
\address{Department of Mathematics and Computer Science, University of Florence, Florence, Italy}
\email{antongiulio.fornasiero@unifi.it}
\author{Noa Lavi}
\address{Department of Mathematics and Physics, Roma Tre University, Rome, Italy}
\email{noa.lavi@unicam.it}

\author{Sonia L'Innocente}
\address{School of Science and Technology, Mathematics Division, University of Camerino, Camerino, Italy}
\email{sonia.linnocente@unicam.it}

\author{Vincenzo Mantova}
\address{School of Mathematics, University of Leeds, LS2 9JT Leeds, United Kingdom}
\email{v.l.mantova@leeds.ac.uk}

\thanks{No data are associated with this article. For the purpose of open access, the authors have applied a Creative Commons Attribution
 (CC BY) licence to any Author Accepted Manuscript version arising from this submission.}
\thanks{The authors were supported by:  PRIN 2017 ``Mathematical logic: models, sets, computability'' 2017NWTM8R (A.F, N.L, S.L.), INdAM-GNSAGA (A.F, S.L.); Engineering and Physical Sciences Research Council ``Model theory of analytic functions'' EP/T018461/1 (V.M.).}



\begin{document}

\begin{abstract}
  A classical tool in the study of real closed fields are the fields $K((G))$ of generalized power series (i.e., formal sums with well-ordered support) with coefficients in a field $K$ of characteristic 0 and exponents in an ordered abelian group $G$. In this paper we enlarge the family of ordinals $\alpha$ of non-additively principal Cantor degree for which $K((\mathbb{R}^{\le 0}))$ admits irreducibles of order type $\alpha$ far beyond $\alpha=\omega^2 $ and $\alpha = \omega^3$ known prior to this work.
\end{abstract}
\maketitle

\section{Introduction}
Rings and fields of power series are classical tools in various areas, such as valuation theory. Call \textbf{generalised power series} a formal sum $b = \sum_\gamma b_\gamma t^\gamma$ where the \textbf{exponents} $\gamma$ vary in an ordered abelian group $G$, the \textbf{coefficients} $b_\gamma$ are taken from some field $K$, and its \textbf{support} $\{\gamma \in G : b_\gamma \neq 0\}$ is well-ordered, namely every nonempty subset has a minimum. It is well known that the collection $K((G))$ of such series forms a field, when equipped with the obvious operations of sum and product \cite{Hah1907}.

We are interested in the irreducible series in the subrings of the form $K((G^{\leq 0}))$ or $Z + K((G^{<0}))$. Such rings appear in different contexts; for instance, $\mathbb{Z} + \mathbb{R}((G^{<0}))$ is always an integer part of the field $\mathbb{R}((G))$, which in turn implies that every real closed field admits an integer part \cite{MR1993}. Regarding irreducibility, the first question was posed by Conway \cite{Con1976}, who conjectured that the series $1 + \sum_n t^{-\frac{1}{n}}$ is irreducible in the ring of omnific integers, which can be written in the form $\mathbb{Z} + \mathbb{R}((G^{<0}))$ (modulo some set-theoretic details which are irrelevant here).

Conway's conjecture was proved by Berarducci \cite{Ber2000}. A crucial part of the argument, first suggested by Gonshor \cite{Gon1986}, is the reduction to the ring $K((\mathbb{R}^{\leq 0}))$. In this paper, we work exclusively in this ring. We refer the reader to \cite{BKK2006,LM} for extensive considerations on how to use irreducibility in $K((\mathbb{R}^{\leq 0}))$ in order to find irreducibles in rings of the form $Z + K((G^{<0}))$.

More irreducible series have been presented in the literature. Let the \textbf{order type} $\ot(b)$ of a power series $b \in K((\mathbb{R}^{\leq 0}))$ be the ordinal number representing the order type of its support. Let $J$ be the ideal of the series that are divisible by $t^\gamma$ for some $\gamma \in \mathbb{R}^{<0}$ (such series cannot be factored into irreducibles, since $t^\gamma = t^{\frac{\gamma}{2}} t^{\frac{\gamma}{2}} = \ldots$). Berarducci's main result reads as follows.

\begin{thm}[{\cite[Thm.\ 10.5]{Ber2000}}]\label{maintB}
  If $b \in K((\mathbb{R}^{\leq 0})) \setminus J$ (equivalently,
  $b \in K((\mathbb{R}^{\leq 0}))$ not divisible by $t^\gamma$ for any
  $\gamma < 0$) has order type $\omega^{\omega^\alpha}$ for some
  ordinal $\alpha$, then both $b$ and $b+ 1$ are irreducible.
\end{thm}

Further irreducible series of order types $\omega^2$ and $\omega^3$ were exhibited in \cite{PS2006} and in \cite{LM2017}; in parallel, Pitteloud \cite{Pit2001} proved that for $b$ of order type $\omega$, $b$ and $b+1$ are prime, answering a question of Gonshor.

We will show that, in an appropriate sense, \emph{most} series are irreducible for a wide class of order types, and give explicit examples of such series.

\begin{defn}
  Let $P_{\alpha} $ denote the set of all series $b \in \KRR$ such that $\ot(b)=\omega^{\alpha}$ and the supremum of the support of $b$ (also denoted by \emph{$\sup(b)$}) is $0$. A series $b$ in $P_{\alpha}$ is said to be \textbf{principal}.
\end{defn}

For instance, $P_0 = K \setminus \{0\}$, and the series $\sum_n t^{-\frac{1}{n}}$ is in $P_1$; on the other hand, $t^{-1} + 1$ and $1 + \sum_n t^{-\frac{1}{n}}$ are not principal since their order types are respectively $2$ and $\omega + 1$, while $\sum_n t^{-1-\frac{1}{n}}$ is also not principal, this time because the supremum of its support is $-1$.

By Berarducci's analysis on principal series, all series in $P_{\alpha}$ satisfy the following properties: for all ordinals $\alpha$, $\beta$, $\gamma$,
\begin{enumerate}
  \item\label{item:Palpha*Pbeta} $P_{\alpha} \cdot  P_{\beta} \subseteq  P_{\alpha \oplus \beta}$ (\cite[Cor.\ 9.9]{Ber2000});
  \item\label{item:Palpha divisors} all divisors of $b \in P_{\alpha} $ are principal (\cite[Cor.\ 4.6]{LM2017});
  \item $R_{\alpha} \coloneqq \{b \in P_{\alpha} : b \, \text{ is reducible}\} \stackrel{\small{\prettyref{item:Palpha divisors}}}{=} \bigcup_{\beta \oplus \gamma = \alpha}^{\beta,\gamma \neq 0} P_{\beta} \cdot P_{\gamma} \subseteq P_\alpha$.
\end{enumerate}
Here $\oplus$, $\odot$ denote Hessenberg's natural (commutative) operations. Note that \prettyref{item:Palpha*Pbeta} is equivalent to stating that $\ot(bc) = \ot(b) \odot \ot(c)$ for all principal series $b$, $c$. One can easily deduce that every principal $b$ of order type $\omega^{\omega^\alpha}$ is irreducible: there are no $\beta, \gamma \neq 0$ such that $\beta \oplus \gamma = \omega^\alpha$, so $R_\alpha = \varnothing$.

\begin{defn}
  We say that $b = \sum_{\gamma}b_{\gamma}t^{\gamma} \in \KRR$ is \emph{random} if one of the following holds:
  \begin{itemize}
    \item $\cl(\supp(b))-\{0\}$ is a $\mathbb{Q}$-linearly independent set;
    \item the tuple $\langle b_{\gamma} : \gamma \in \supp(b) \rangle $ is algebraically independent over $\mathbb{Q}$.
  \end{itemize}
\end{defn}

In the current work we prove the following:
\begin{thm} \label{printhm}
  Let $\alpha$ be an ordinal of the following forms:
  \begin{enumerate}
    \item $\alpha = \omega^{\beta_1} + \dots + \omega^{\beta_n}$ where $\beta_1 \geq \beta_2 \ge \beta_3$ and $\beta_3 > \dots >\beta_n$.
          \item\label{item:three-summands} $\alpha = \omega^{\beta_1} + \omega^{\beta_2} + k$ where $\beta_1 \ge \beta_2$ and $k \in \mathbb{N}$.
  \end{enumerate}
  If $b \in P_{\alpha}$ is random, then $b$ is irreducible, and so is $b+r$ for any principal series $r$ of order type less than $\omega^\alpha$.

  If $\alpha$ is of the form \prettyref{item:three-summands}, then there are no principal series $p_1, \ldots , p_m $, $q_1, \ldots , q_m $, $r$ such that $b = \sum_{i=1}^m p_iq_i + r$, where $\ot(r) < \omega^{\alpha}$ and for every $1 \le i \le m $ we have $1<\ot(p_i)\le\ot(q_i)$, $\ot(p_iq_i)=\omega^{\alpha}$ (that is, $p_iq_i \in R_\alpha$).
\end{thm}
\begin{cor}
  Let $b$ in $P_n$, where $n \in \Nb$. If $b$ is random, then $b$ is irreducible.
\end{cor}
\begin{cor}
  Let $\alpha$ be as in Theorem~\ref{printhm}\prettyref{item:three-summands}. Then $\operatorname{Span}_K(R_{\alpha}) $ is infinite co-dimensional in $\operatorname{Span}_K(P_{\alpha})$ as a $K$-vector space.
\end{cor}

\begin{rem}
  As a special case, we find irreducible principal series of order types $\omega^2$ and $\omega^3$, as in respectively \cite{PS2006} and \cite{LM2017}.
\end{rem}
Combining the above result with techniques from \cite{LM}, we are also able to produce irreducible series that are not principal, at least for certain order types.
\begin{thm} \label{genthm}
  Let $b \in \KRR$ be such that $\sup(b)=0$ and $\ot(b)=m\omega^{\alpha}+\beta$, where $\alpha$ is as in Theorem~\ref{printhm}, $m \in \mathbb{N}$ non-zero, and $\beta < \omega^{\alpha}$. If $b$ is random, then $b$ is irreducible and so is $b+r$ for any $r$ such that $\ot(r)<\omega^{\alpha} $ and $\sup(b+r)=0$.
\end{thm}

In \prettyref{sec:hered-indep} we define an independence relation which is a generalization of randomness, and can substitute it in all of the above.

We conjecture that in fact all random series are irreducible, and more precisely, we expect the following to hold.

\begin{conj}
  Suppose that $b \in \KRR$ is random and $\sup(b) = 0$. Then $b$ is irreducible and so is $b+r$ for any series $r$ with $ot(r)< \omega^{\deg(b)}$.
\end{conj}

Here $\deg(b)$ is the maximum ordinal $\alpha$ such that $\omega^\alpha \leq \ot(b)$.

\begin{conj}
  $\operatorname{Span}_K(R_{\alpha})$ is infinite co-dimensional in $\operatorname{Span}_K(P_{\alpha})$ as a $K$-vector space for any $\alpha >0$.
\end{conj}

\section{Preliminaries}
We assume that the reader has familiarity with the class $\on$ of ordinal numbers and the classical (non-commutative) operations on them, but we will give a minimal account of Hessenberg's commutative operations.

Recall that for all $\alpha \in \on$, there is a maximum $\beta$ such that $\omega^\beta \leq \alpha$ and a unique $\gamma$ such that $\alpha = \omega^\beta + \gamma$. By repeating the argument, we find a unique finite sequence $\beta_1 \geq \beta_2 \geq \ldots \geq \beta_n \geq 0$ of ordinals such that
\begin{equation}
  \label{eq:cantor-normal-form}
  \alpha = \omega^{\beta_1} + \ldots + \omega^{\beta_n}.
\end{equation}
The expression on the right-hand side is called \textbf{Cantor normal form} of $\alpha$. We let $\beta_1$ be the \textbf{Cantor degree of $\alpha$}, denoted by $deg(\alpha)$.

Given $\alpha = \omega^{\gamma_1} + \omega^{\gamma_2} + \ldots + \omega^{\gamma_n}$ and $\beta = \omega^{\gamma_{n+1}} + \omega^{\gamma_{n+2}} + \ldots + \omega^{\gamma_{n+m}}$ in Cantor normal form, let $\pi$ be a permutation of the integers $1, \ldots, n+m$ such that $\gamma_{\pi(1)} \geq \ldots \geq \gamma_{\pi(n+m)}$. Then $\alpha \oplus \beta$ is defined to be $\omega^{\gamma_{\pi(1)}}+ \ldots + \omega^{\gamma_{\pi(n)}} $, and $\alpha \odot \beta $ to be $\bigoplus_{1\le i \le n,n+1 \le j \le n+m}\omega^{\gamma_i \oplus \gamma_j} $. Note that by definition $\omega^{\alpha} \odot \omega^{\beta} = \omega^{\alpha \oplus \beta}$ for all $\alpha, \beta \in \on$. Moreover, $\omega^{\beta_1} + \cdots + \omega^{\beta_n} = \omega^{\beta_1} \oplus \cdots \oplus \omega^{\beta_n}$ if (and only if) the left hand side is in Cantor normal form. These are \textbf{Hessenberg's natural sum and product}.

We remark that in Cantor normal form, Hessenberg's operations can be interpreted as sum and product for polynomials in the variables $\omega^{\omega^\alpha}$, for $\alpha \in \on$. We summarise this consideration with the isomorphism
\[ (\on, \oplus, \odot) \cong \left(\Nb\left[\omega^{\omega^0},\omega^{\omega^1},\dots,\omega^{\omega^\alpha},\dots\right], +, \cdot\right). \]

An ordinal is \textbf{additively principal} if it cannot be written as a sum of two strictly smaller ordinals (equivalently, it is of the form $\omega^\alpha$) and \textbf{multiplicatively principal} if it cannot be written as a product of two strictly smaller ordinals (equivalently, it is of the form $\omega^{\omega^\alpha}$).

\begin{defn}[{\cite[Def.\ 3.3.6]{LM}}]\label{def:normal-form}
  Given $b \in K((\Rb))$, we call the sum
  \[ b = b_1t^{x_1} + \dots + b_nt^{x_n} \]
  the \textbf{normal form} of $b$ when:
  \begin{itemize}
    \item $x_1 \le \dots \le x_n$;
    \item $b_i$ is principal for all $i = 1, \dots, n$;
    \item $\ot(b_1) \le \dots \le \ot(b_n)$;
    \item $x_i + \supp(b_i) < x_{i+1} + \supp(b_{i+1})$ for all $i = 1, \dots, n-1$.
  \end{itemize}
\end{defn}
By \cite[Prop.\ 3.3.7]{LM}, every $b \in \KRR $ has a unique normal form.

Let $J$ be the (proper) ideal of $\KRR$ generated by the series of the form $t^x$ for $x < 0$. Given $b \in \KRR$, the \textbf{ordinal value} of $b$ (\cite[p.\ 558]{Ber2000}) is
\[ v_J(b) \coloneqq \begin{cases}
    0                                      & \text{if } b \in J,                   \\
    1                                      & \text{if } b \in (J + K) \setminus J, \\
    \min\{\ot(c) : c \equiv b \mod J + K\} & \text{otherwise}.
  \end{cases} \]
The keystone of \cite{Ber2000} is that $v_J$ is a multiplicative semi-valuation.

\begin{fact}[{\cite[Lem.\ 5.5,\ Thm.\ 9.7]{Ber2000}}]
  For all $b, c \in \KRR$ we have:
  \begin{itemize}
    \item $v_J(b + c) \leq v_J(b) \oplus v_J(c)$;
    \item $v_J(bc) = v_J(b) \odot v_J(c)$;
    \item $v_J(b) = 0$ if and only if $b \in J$.
  \end{itemize}
\end{fact}
The values of $v_J$ are all of the form $\omega^\alpha$ (\cite[Rem.\ 5.3]{Ber2000}), thus we add the following notation: let $\deg_J(b)$ denote the Cantor degree of $v_J(b)$, where by convention we set $\deg_J(b) = -\infty$ when $b \in J$, and $\omega^{-\infty} = 0$. Thus we have $v_J(b) = \omega^{\deg_J(b)}$, and we may rephrase the above properties as:
\begin{itemize}
  \item $\deg_J(b+c) \leq \max\{\deg_J(b),\deg_J(c)\}$;
  \item $\deg_J(bc) = \deg_J(b) \oplus \deg_J(c)$;
  \item $\deg_J(b) = -\infty$ if and only if $\deg_J(b) = -\infty$.
\end{itemize}

We note that to prove $\ot(bc) = \ot(b) \odot \ot(c)$ for $b$, $c$ principal, it was stated with no proof in \cite{Ber2000} that there are ``few cancellation'' in the product of two elements of $\KRR$.
We observe that not only the above is not required (and in fact not used) in order to prove the multiplicativity of $v_J$, it is neither true.

\begin{exa}
  Let $S= \{s_1, s_2, \ldots \, \}$ be a strictly increasing sequence of $\mathbb{Q}$-linearly independent real numbers such that $\sup(S)=0$. Let us define two series $b = \sum_{n \in \mathbb{N}} t^{s_n}$ and $c =  \sum_{n \in \mathbb{N}}(-1)^n  t^{s_n}$. If $m$ and $n$ have different parities, the coefficient of $t^{s_m+s_n}$ in $bc$ is $(-1)^m + (-1)^n = 1 - 1 = 0$. Consider $T= \{ s_m + s_n : m - n \text{ is odd}\}$, then for every $p \in T$ we have $p \notin \supp(bc)$, thus cancellations occurs on a set of order type $\ot(T)=\omega^2$, which is the same as $\ot(b) \odot \ot(c)$.
\end{exa}

The following notions are fundamental in our work.
\begin{defn}
  Given
  $b = \sum_{\beta} b_x t^{x} \in K((\mathbb{R}^{\leq 0}))$
  and $\gamma \in \mathbb{R}^{\leq 0}$, we define:
  \begin{itemize}
    \item the \textbf{truncation} of $b$ at $\gamma$ is
          $b_{\vert \gamma} \coloneqq \sum_{x \leq \gamma } b_x t^{x}$,
    \item the \textbf{translated truncation} of $b$ at $\gamma$ is
          $b^{\vert \gamma} \coloneqq t^{-\gamma} b_{\vert \gamma}$.
  \end{itemize}
  The equivalence class $b^{\vert \gamma} + J$ is the \textbf{germ of
    $b$ at $\gamma$}.
\end{defn}

\begin{defn}
  Let $b \in K((\mathbb{R}^{\le 0})) \setminus J$. The \textbf{$J$-critical point} of $b$, denoted by $\crit_J(b)$, is the minimal $\gamma \in \supp(b)$ such that for every $\gamma < \delta < 0$ we have $v_J(b^{|\delta}) < v_J(b)$.
\end{defn}
We note that if $b \in P_{\alpha}$, then $\crit_J(b)=\min(\supp(b))$.
\begin{rem}
  One should not confuse this with the notion of \emph{critical point} $\crit(b)$, which is the minimum $\gamma \in \mathbb{R}^{\leq 0}$ such that $v_J(b^{|\gamma})$ is maximum possible \cite[Def.\ 10.2]{Ber2000}. For example for $b=\sum_{n \in \Nb} t^{-1 - \frac{1}{n+1}} + \sum_{n \in \Nb} t^{-\frac{1}{n+1}} + \sum_{n \in \Nb} t^{-2 - \frac{1}{n+1}}$, we have that $v_J(b)=\omega=v_J(b^{|-1})=v_J(b^{|-2})$, and one can verify that $\crit(b)=-2$ and $\crit_J(b)=-1$ .
\end{rem}
Given $b \in \KRR $ with $v_J (b) > 1$, we know that $v_J (b)$ has the form $\omega^\beta$ for some ordinal $\beta >0$. From the Cantor normal form of $\beta$ it follows that $v_J(b)$ can be written uniquely as a product $\omega^{\omega^{\beta_1}} \cdots\omega^{\omega^{\beta_n}}$, where $\beta_1 \geq \beta_2 \geq \ldots  \geq \beta_n$. We define ({\cite[Def.\ 6.4]{Ber2000}}):
\begin{enumerate}
  \item $ v_J^p (b) \coloneqq \omega^{\omega^\beta_n}$ = the \textbf{principal value} of $b$,
  \item $ v_J^r (b) \coloneqq \omega^{\omega^{\beta_1}} \cdots\omega^{\omega^{\beta_{n-1}}}$ = the \textbf{residual value} of $b$.
\end{enumerate}
For instance, if $v_J ( b ) = \omega^3$, then $v_J^p ( b ) = \omega$ and $v_J^r ( b ) = \omega^2$. Note that $v_J(b) = v_J^r(b)v_J^p(b)$. For the sake of readability, we introduce the following notation to remove the base $\omega$.
\begin{enumerate}
  \item $\deg_J^p (b) \coloneqq \omega^{\beta_n}$,
  \item $\deg_J^r (b) \coloneqq \omega^{\beta_1} + \cdots + \omega^{\beta_{n-1}}$.
\end{enumerate}
We have $\deg_J(b) = \deg_J^r(b) + \deg_J^p(b)$. Note that $\deg_J^p(b)$ is \textbf{additively principal}: the ordinals strictly less than $\deg_J^p(b)$ are closed under addition.

\begin{defn}
  Let $b \in K((\mathbb{R}^{\leq 0}))$. We say that $\gamma \in \cl(\supp(b)) - \{0\}$ is a \textbf{big point} of $b$ if $\deg_J(b^{\vert\gamma})\ge \deg_J^r(b)$ and $\gamma > \crit_J(b)$. The set of all big points of $b$ is denoted by $ \BigP (b)$.

  More generally, let $\BigP^{\alpha}(b)$ denote all the numbers $\gamma \in \cl(\supp(b)) - \{0\}$ such that $\deg_J(b^{|\gamma}) \ge \alpha $ and $\gamma > \crit_J(b)$.

  We say that $\gamma \in \cl(\supp(b)) - \{0\}$ is a \textbf{residual point} of $b$ if $\deg_J(b^{\vert\gamma})=\deg_J^r(b) $. The set of all these point $\Res(b)$ was defined in \cite[Def.\ 6.6]{Ber2000} (where it is called $X(b)$).
\end{defn}

\begin{rem}
  By construction, the big points of a series must accumulate to $0$. Moreover, they are accumulation points of $\cl(\supp(b))$ as soon as $\deg_J^r(b) > 0$.
\end{rem}

It turns out that translated truncations behave like a sort of
`generalised coefficients', as they satisfy the following equation.

\begin{prop}[{\cite[Lem.\ 7.5(2)]{Ber2000}}] \label{conv}
  For all $a, b \in K((\mathbb{R}^{\leq 0}))$ and
  $\gamma \in \mathbb{R}^{\leq 0}$ we have:
  \[
    {(ab)}^{\vert\gamma} \equiv \sum_{\delta + \varepsilon = \gamma}
    a^{\vert\delta} b^{\vert\varepsilon} \mod J \quad
    \textrm{\emph{(convolution formula)}}.
  \]
\end{prop}

And one further more may obtain a kind of a Leibniz rule.
\begin{prop} [{\cite[Lem.\ 7.7]{Ber2000}}] \label{leib}
  Let $b,c \in \KRR$ such that $\deg_J^p(b) \le \deg_J^p(c)$. Then for every $\gamma$ sufficiently close to $0$ we have $(bc)^{|\gamma} = b^{|\gamma}c+c^{|\gamma}b +r $ where $\deg_J(r) < \deg_J^r(b) \oplus \deg_J(c) < \deg_J(bc)$.
\end{prop}
\begin{defn}
  For every $\alpha$ we let $J_{\alpha} \coloneqq \{b \in \KRR : \deg_J(b) < \alpha\} $. $J_{\alpha}$ is a $K$-vector space, and for $\alpha$ additively principal also a ring, which would allow the use of $J_{\alpha}$-linear combinations in our proofs.
\end{defn}

In this work we also make use of the \emph{degree valuation}, in order to prove irreducibility results for non-principal elements.
\begin{defn}[{\cite[p.\ 5]{LM}}]
  Given $b \in \KRR$ with $b \neq 0$, let the \textbf{degree} of $b$, denoted by $\deg(b)$, be the Cantor degree of $\ot(b)$. We also let $\deg(0) \coloneqq -\infty$.
\end{defn}

For instance, $t^{-\sqrt{2}} + t^{-1} + 1$ has degree $0$, while $\sum_{n \in \Nb} t^{-1 - \frac{1}{n+1}} + \sum_{n \in \Nb} t^{-\frac{1}{n+1}}$ has degree $1$ because its order type is $\omega + \omega$, and $\sum_{(m,n) \in \Nb} t^{-\frac{1}{(n+1)(m+1)}}$ has degree $2$, as its order type is $\omega^2$. In \cite{LM} it is proved by the 3rd and the 4th authors that the degree is an \textbf{multiplicative valuation} in the following sense:

\begin{fact}[{\cite[Thm.\ D]{LM}}]
  For all non-zero $b, c \in \KRR$,
  \begin{enumerate}
    \item $\deg(b + c) \leq \max\{\deg(b), \deg(c)\}$ (ultrametric inequality);
    \item $\deg(bc) = \deg(b) \oplus \deg(c)$ (multiplicativity);
    \item $\deg(b) = -\infty$ if and only if $b = 0$.
  \end{enumerate}
\end{fact}

This makes the degree is quite similar to $\deg_J$, but in a sense more precise: $\deg(b) = -\infty$ only when $b = 0$, whereas $\deg_J(b) = -\infty$ for all $b \in J$.

\begin{rem}
  \label{rem:vJ-deg-inequality}
  Note that by construction $v_J(b) \leq \ot(b)$, thus $\deg_J(b) \leq \deg(b)$.
\end{rem}

The above inequality is often strict, and not just for series $b \in J$, for instance
\[b = \sum_{(n,m) \in \Nb^2} t^{-1 - \frac{1}{(n+1)(m+1)}} + \sum_{n \in \Nb} t^{-\frac{1}{n+1}} \]
has degree $2$ while $\deg_J(b)=1$.

We also remark that all well ordered subsets of $\mathbb{R}$ are countable, thus $\ot$, $v_J$, $\deg_J$ and $\deg$ all take values below $\omega_1$, the first uncountable ordinal. One can easily verify that all values below $\omega_1$ are in the images of $\ot$, $\deg_J$, $\deg$ (but not $v_J$ which is always of the form $\omega^\alpha$).

\begin{defn}
  For $\nu \in \{\deg, \deg_J\} $ let $\rv_{\nu}(b)=\rv_{\nu}(c) $ if and only if $\nu(b-c)<\nu(b)$ or $b = c$, and let $\RV_\nu$ be the quotient of $K((\mathbb{R}^{\leq 0}))$ by this equivalence relation. For the sake of notation, we write $\rv_J$, $\RV_J$ for respectively $\rv_{\deg_J}, \RV_{\deg_J}$.
\end{defn}
For both $\nu$ as above, $\RV_\nu$ is the union over $\alpha < \omega_1$ of the quotients $\RV_\nu^\alpha \coloneqq \{b : \nu(b) \leq \alpha\} / \{c : \nu(b) < \alpha\}$. In particular, each $\RV_\nu^\alpha$ is naturally a $K$-vector space, $\RV_\nu^\alpha \cap \RV_\nu^\beta = \{\rv_\nu(0)\}$ for $\alpha \neq \beta$, and notably  $\RV_J^{\alpha} = J_{\alpha +1}/ J_{\alpha}$.

Moreover, for $b$, $c$ principal, we have $\deg(b - c) < \deg(b) = \deg_J(b)$ if and only if $\deg_J(b - c) < \deg_J(b)$. This means that $\rv_J$ and $\rv_{\deg}$ agree on the principal series. On the other hand, by definition of $\deg_J$, for every $b$ there is a principal $c$ such that $\rv_J(b) = \rv_J(c)$: in this case we simply write $\rv(b)$. Therefore, each $\RV_J^\alpha$ embeds into $\RV_{\deg}^\alpha$ as $K$-vector space and coincides with the image of the principal series under the quotient map. However, $\RV_{\deg}^\alpha$ is much richer, for instance $\RV_J^0 \cong K$ while $\RV_{\deg}^0 \cong K[\mathbb{R}^\leq]$, the space of the series with finite support. Thanks to $\RV_{\deg}$ and its properties proved in \cite{LM} we are able to find irreducible elements of degree $\alpha$ which are not principal, as we shall see in the following sections.

\section{Hereditary \texorpdfstring{$\rv_J$}{rv\textunderscore{}J-independence}}
\label{sec:hered-indep}

Given $b_1, \dots, b_n$ with $\deg_J(b_1) = \deg_J(b_i)$ for every $i = 1, \dots, n$, we give the following definition by induction on $\deg_j(b_1)$. The series $b_1, \dots, b_n$ are said to be \textbf{hereditarily $\rv_J$-independent}, written $\QQ(b_1, \dots, b_n)$, if:
\begin{enumerate}
  \item\label{axiom:indep}  $(\rv_J(b_i))_{1 \le i \le n}$ are $K$-linearly independent;
  \item\label{axiom:inductive-indep} when $\deg_J(b_1) \neq \deg_J^p(b_1)$, there exists some $\delta  < 0$ such that for all $\alpha$ with  $\deg_J^r(b_1) \leq \alpha < \deg_J(b_1)$, for all $\gamma_{1,1}, \ldots, \gamma_{n, 1}, \ldots , \gamma_{n, m(n)} \geq \delta$ with $\deg_J(b_i^{|\gamma_{i,j}})=\alpha$ and $\gamma_{i,j} \neq \gamma_{i,j'}$ whenever $j \neq j'$, we have \[ \QQ(b_1^{|\gamma_{1,1}},\ldots , b_1^{|\gamma_{1,m(1)}}, \ldots , b_n^{|\gamma_{n,1}}, \ldots , b_n^{|\gamma_{n,m(n)}}). \]
\end{enumerate}
Note that the above definition is obviously well founded.

By \cite[Prop.\ 5.5.1]{LM}, we know that every $b$ has a maximal divisor in $K[\mathbb{R}^{\le}]$ (the series with finite support), which is unique up to a multiplication by an element in $K$. Clearly, for every $b$ there is a unique maximal divisor $d$ such that the coefficient of $ t^{\sup(d)}$ is $1$, denoted by $p(b)$.

The following allows us to obtain irreducibility for non-principal elements using irreducibility results for principal elements.
\begin{lem} \label{indrv}
  Let $b \in \KRR$ such that $\sup(b)=0$. Suppose $b= \sum_{i=1}^m b_it^{\gamma_i}+r$ where $b_i \in P_{\deg(b)} $ for $1 \le i \le m$, $\gamma_1 < \ldots < \gamma_m $ and $\deg(r)< \deg(b) $. Suppose also that $\{ \rv(b_1) , \ldots , \rv(b_m) \}$ are linearly independent. Then $p(b)=1$.
\end{lem}
\begin{proof}
  By \cite[Prop.\ 5.3.1]{LM} we have that if $\{C_i\}_{i \in I}$ is a base for $\RV_J^{\alpha} $ then for every $B \in \RV_{\deg}^{\alpha}$ we have that $B = q_iC_i $ where $q_i \in K[\mathbb{R}^\leq]$ and $q_i \neq 0$ for finitely many $i \in I$. As $\{\rv(b_1), \ldots , \rv(b_m) \}$ is a part of a base of $\RV_J^{\alpha}$ then by the algorithm described in \cite[Rem.\ 5.4.7]{LM} we have that $t^{\gamma_m}=\operatorname{gcd}(t^{\gamma_1}, \ldots , t^{\gamma_m})$ is a maximal divisor of finite support of $\rv(b)$, hence $p(\rv(b))=t^{\gamma_m}$. As $\rv(p(b))=p(b)$ it follows that $p(b)$  divides $\rv(b)$, hence by maximality it divides also $p(\rv(b)) $. As $\sup(b)=0$, $p(b)$ must be a unit.
\end{proof}
We are now ready to prove irreducibility for non-principal series:
\begin{prop} \label{genir}
  Let $\alpha$ be such that for every $c \in P_{\alpha}$ with $\QQ(c)$ we have that $c$, $\rv(c)$ are irreducible. Let $b= \sum_{i=1}^m b_it^{\gamma_i}+r$ where $b_1, \ldots , b_m \in P_{\alpha} $, $\gamma_1 < \ldots < \gamma_m $ and $\ot(r)<\omega^{\alpha} $. If $\QQ(b_{1}, \ldots , b_{m}) $ holds then $b$ is irreducible.
\end{prop}
\begin{proof}
  By Lemma \ref{indrv} we know that $p(b)=1$, as $\{\rv(b_1), \ldots , \rv(b_m)\} $ are linearly independent. By \cite[Lem.\ 7.1.1]{LM} it is enough to prove that $\frac{\rv(b)}{p(\rv(b))}$ is irreducible. Let $B=\frac{\rv(b)}{t^{\gamma_m}}$ and suppose that $AC=B$. Then there exist some $a$, $c$, $b'$ such that $\rv(a)=A$, $\rv(b')=B$, $\rv(c)=C$ and $ac=b'$. Then $ac = d \mod J$ for some $d \in P_{\alpha} $ such that $\rv(d)=\rv(b_{m})$ (equivalently, $\rv_J(ac) = \rv_J(b_m)$). As $\rv(d)$ is irreducible, this implies that, without loss of generality, $a \in J+K$. Let $a = \sum_{i=1}^s a_it^{\delta_i}+a_0 $ be the normal form of $a$, then $a_0 \in K^{\times}$. As $p(B)=1$ then also $p(A)=1$. Hence, we must have also $\deg_J(a_i) = 0$ for every $1 \le i \le s$. Hence $a \in K[\mathbb{R}^{\le}]$, and therefore $ a \in K$. Hence, $B$ is irreducible.
\end{proof}
\begin{rem}
  In the proof above we used only \prettyref{axiom:indep} -- linear independence in $\RV_J$. In the next section where we use induction arguments we will make heavy use of \prettyref{axiom:inductive-indep}.
\end{rem}
We say that $b_1 , \ldots , b_m \in \KRR$ are \emph{mutually random} if one of the following holds:
\begin{itemize}
  \item  $\cl(\supp(b_i))\cap \cl(\supp(b_j)) = \{0\} $ for every $1 \le i \neq j \le m $ and \\ $\bigcup_i \cl(\supp(b_i)) - \{0\} $ is $\mathbb{Q} $-linearly independent set.
  \item $ \langle {b_i}_{\gamma} : 1 \le i \le m, \gamma \in \supp(b_i) \rangle  $ is algebraically independent over $\mathbb{Q} $.
\end{itemize}
We show now that $b_1, \ldots , b_m $ mutually random implies $\QQ(b_1, \ldots , b_m)$.
\begin{prop} \label{algind}
  Let $b_1, \ldots , b_n$ be such that $ \langle {b_i}_{\gamma} : 1 \le i \le n, \gamma \in \supp(b_i) \rangle  $ is algebraically independent over $\mathbb{Q} $. Suppose that $\deg_J(b_i)=\alpha \ge 1$ for every $1 \le i \le n$. Then $\rv_J(b_1), \ldots , \rv_J(b_n)$ are linearly independent.
\end{prop}
\begin{proof}
  Suppose by contradiction that we have $k_1, \ldots k_n \in K$ not all zero such that $\sum_{i=1}^n k_ib_i = r \in J_{{\alpha}}$. If we replace each coefficient $b_{i\gamma}$ with $0$ for $\gamma \in \supp(r)$, we find $u_1, \ldots, u_n$ all with the \emph{same} support $\bigcup_{i=1}^n \supp(b_i) \setminus \supp(r)$ of order type ${\omega^{\alpha}}$ and such that $\sum_{i=1}^n k_iu_i = 0$. Moreover, the coefficients of $u_1, \dots, u_n$ satisfy the same algebraic independence assumption as for $b_1, \dots, b_n$. Let $\gamma_1, \ldots \gamma_n \in \supp(u_1)$. We define $A \in M_{n\times n }(K)$ such that for every $1\le i,j \le n$ we have $A[i,j]=u_{j\gamma_i}$, and let $\overline{v} = [ k_1, \ldots , k_n ] \neq \overline{0}$. Then we have $A\overline{v}=\overline{0}$ and hence $\det(A)=0$. Hence, the elements of $A$ are not algebraically independent, a contradiction.
\end{proof}
\begin{cor}
  If $ \langle {b_i}_{\gamma} : 1 \le i \le m, \gamma \in \supp(b_i) \rangle  $ is algebraically independent over $\mathbb{Q} $, and $\deg_J(b_i)=\deg_J(b_1)$ for every $2 \le i \le n $, then $\QQ(b_1 , \ldots , b_n)$ holds.
\end{cor}

\begin{prop}\label{prop:randomness-implies-indep}
  If $\deg_J(b_1)=\deg_J(b_i) > 0$ for $2 \le i \le n$,  $\cl(\supp(b_i))\cap \cl(\supp(b_j)) = \{0\} $ for every $1 \le i \neq j \le m $ and $\bigcup_i \cl(\supp(b_i)) - \{0\} $ is $\mathbb{Q} $-linearly independent set, then $\QQ(b_1, \ldotp , b_n)$ holds.
\end{prop}
\begin{proof}
  For simplicity assume $b_1, \ldots , b_n \in P_{\deg_J(b_1)}$. Since the closures of the supports have pairwise intersection $\{0\}$, the supports are pairwise disjoint, so $\rv_J(b_1),\dots,\rv_J(b_n)$ must be $K$-linearly independent, proving \prettyref{axiom:indep}.

  Suppose $\deg_J(b_1)\neq \deg_J^p(b_1)$. Let \[S=\{\gamma_{1,1} , \ldots , \gamma_{1,{m(1)}}, \ldots , \gamma_{n,1} , \ldots , \gamma_{n,{m(n)}} \} \] as in \prettyref{axiom:inductive-indep}, and suppose by contradiction that there are
  \[k_{1,1}, \ldots , k_{1,{m(1)}}, \ldots , k_{n,1} , \ldots , k_{n,{m(n)}} \in K\]
  not all zero such that $b=\sum_{i=1}^n \sum_{j=1}^{m(i)} k_{i,j}b_i^{|\gamma_{i,j}} \in J_{\deg_J(b_1)} $.

  By comparing the ordinal values, there are $i$, $j$ and $\delta < 0$ such that $\delta$ is in $\supp(b_i^{|\gamma_{i,j}})$, but not in $\supp(b)$. For $\delta$ to cancel, there must be $(i',j') \neq (i,j)$ such that $\delta \in \supp(b_{i'}^{|\gamma_{i',j'}})$. But then $\beta - \gamma_{i,j} = \delta = \beta' - \gamma_{i',j'}$ for some $\beta \in \supp(b_i)$ and $\beta' \in \supp(b_{i'})$. If $i = i'$, then $\gamma_{i,j} \neq \gamma_{i,j'}$; if $i \neq i'$, since the closure of the supports are disjoint, then $\BigP(b_i) \cap \BigP(b_{i'}) = \varnothing$, hence again $\gamma_{i,j} \neq \gamma_{i',j'}$. Therefore, $\{\beta,\beta',\gamma_{i,j},\gamma_{i',j'}\}$ is not linearly independent, a contradiction.
\end{proof}
Note that in the above proof, we may further require $\delta$ to be isolated within the set $\bigcup_{i=1}^n\cl(\supp(b_i^{|\gamma_{i,j}}))$, since the isolated points themselves have order type $v_J(b_1^{|\gamma_{1,j}})$, thus most of them must cancel out. In particular, we can find $\beta$, $\beta'$ isolated. Moreover, recall that each $\gamma_{i,j}$ is in $\BigP(b_i)$, thus when $\deg_J(b_1) \neq \deg_J^p(b_1)$, that is $\deg_J^r(b_1) > 0$, it is a limit point.

As the contradiction above arises from the equality $\gamma_{i,j} - \gamma_{i',j'} = \beta - \beta'$, we may weaken the condition and allow dependence between either the isolated points or the big points. In particular, to verify \prettyref{axiom:inductive-indep} it suffices to test the independence on the points of $\supp(b_i)$ only, rather than all of $\cl(\supp(b_i))$. Likewise, for \prettyref{axiom:indep}, it is sufficient to check that the supports of the $b_i$'s pairwise intersect only on accumulation points, or alternatively that the $b_i$'s do not have big points in common, again pairwise. Let $\operatorname{Lim}(b)$ be the set of all accumulation points in $\cl(\supp(b))$.

\begin{cor}
  Suppose that $\deg_J(b_1) = \deg_J(b_i) > 0$ for $2 \le i \le n$ and:
  \begin{itemize}
    \item $(\supp(b_i) \setminus \operatorname{Lim}(b_i)) \cap (\supp(b_j) \setminus \operatorname{Lim}(b_j)) = \varnothing$ for $1 \le i\neq j \le n $;
    \item $\bigcup_{i=1}^n \supp(b_i) \setminus \operatorname{Lim}(b_i)$ is $\mathbb{Q}$-linearly independent over $\operatorname{Span}_{\mathbb{Q}}(\bigcup_{i=1}^n \BigP(b_i))$.
  \end{itemize}
  Then $\QQ(b_1, \ldotp , b_n)$ holds.
\end{cor}
\begin{cor}
  If $\deg_J(b_1)=\deg_J(b_i) > 0$ for $2 \le i \le n$ and:
  \begin{itemize}
    \item $\BigP(b_i) \cap \BigP(b_j) = \{0\} $ for every $1\le i\neq j \le n$;
    \item $\bigcup_{i=1}^n \BigP(b_i)$ is $\mathbb{Q}$-linearly independent over $\operatorname{Span}_{\mathbb{Q}}(\bigcup_{i=1}^n \supp(b_i)\setminus\operatorname{Lim}(b_i))$.
  \end{itemize}
  Then $\QQ(b_1, \ldotp , b_n)$ holds.
\end{cor}

\section{Irreducible principal elements}
We obtain our results using induction on the Cantor normal form of $\alpha$ for elements of ordinal value $\omega^{\alpha} $.
\begin{defn}
  We let $A_{\alpha}$ be the $K$-vector space generated by $J_{\alpha}$ and all the elements of the form $bc$ where $b$, $c$ are principal not in $K$ and $\deg_J(bc)=\alpha$; in other words, $A_\alpha \coloneqq J_\alpha + \operatorname{Span}_K(R_\alpha)$.
\end{defn}
\begin{rem}
  As observed in the introduction, if $\alpha$ is additively principal, then $R_\alpha = \varnothing$, hence $A_{\alpha}=J_{\alpha}$.
\end{rem}
We define the following property for $\alpha < \omega_1$.
\begin{itemize}
  \item[$(*)_{\alpha} $] For every $n \in \mathbb{N} $ and $b_1, \ldots , b_n \in P_{\alpha} $ such that $\QQ(b_1, \ldots , b_n) $ holds we have that $b_1, \ldots , b_n $ are $K$-linearly independent over $A_{\alpha}$.
\end{itemize}
Note in particular that if the above property holds, then $b_1, \dots, b_n$ are irreducile. We shall verify that $(*)_{\alpha}$ implies $(*)_{\alpha + \beta}$ for certain choices of $\beta$, and thus prove that the property holds for many ordinals. First, we observe the following easy base case.

\begin{prop}\label{prop:base}
  $(*)_{\omega^\beta}$ is true for every $\beta$.
\end{prop}
\begin{proof}
  When $\alpha = \omega^\beta$, the condition $\QQ(b_1, \ldots , b_n) $ only states that $\rv_J(b_1), \dots, \rv(b_n)$ are linearly independent, which means exactly that $b_1, \dots, b_n$ are linearly independent over $J_\alpha = A_\alpha$.
\end{proof}

\subsection{The successor case}
\begin{prop} \label{suc}
  $(*)_{\alpha} \Rightarrow (*)_{\alpha+1} $.
\end{prop}
\begin{cor}\label{cor:omega-beta+n}
  Let $\alpha=\omega^{\beta}+n$ where $n \in \mathbb{N}$.
  \begin{itemize}
    \item For every $b \in P_{\alpha}$ with $\QQ(b)$ we have that $b$, $\rv(b)$ are irreducible and cannot be represented as a sum of reducible elements in $\RV_J^{\alpha}$.
    \item $\rv(A_{\alpha})$ has infinite co-dimension as a $K$-vector subspace of $\RV_J^{\alpha}$.
  \end{itemize}
\end{cor}
\begin{rem}
  A special case would be $\alpha = n$.
\end{rem}
\begin{proof}[Proof of \prettyref{cor:omega-beta+n}]
  We show by induction that $(*)_{\alpha}$ holds for every $\alpha = \omega^{\beta}+n$.
  The induction base is $\alpha=\omega^{\beta}$, thus \prettyref{prop:base}. The induction step is Proposition~\ref{suc}. Hence, for ever $b \in P_{\alpha}$ where $\alpha = \omega^{\beta}+n$ with $\QQ(b)$, we have $b \notin A_{\alpha}$ and for every $r \in J_{\alpha}$ we have $b-r \notin A_{\alpha}$, hence $b$, $\rv(b)$ are irreducible.
\end{proof}

\begin{proof}[Proof of Proposition~\ref{suc}]
  The successor stage is rather simple, because there are no ordinals between $\alpha$ and $\alpha+1$.

  Suppose by contradiction that there are $b_1, \ldots , b_n \in P_{\alpha+1}$ and $\lambda_1, \ldots , \lambda_n \in K^{\times}$ such that $b=\sum_{i=1}^n \lambda_ib_i \in A_{\alpha+1}$. Note that by the assumption $\QQ(b_1,\dots,b_n)$, we have $b \in P_{\alpha+1}$. Write $b = \sum_{i=1}^m p_iq_i +r $ where $\deg_J(q_i) \ge \deg_J(p_i) > 0$, $\deg_J(p_iq_i)= \alpha+1$ and $r \in J_{\alpha+1}$. For every $\gamma \in \Res(b)$ we have by Proposition~\ref{leib}
  \[b^{|\gamma} = \sum_{i=1}^m p_i^{|\gamma}q_i + q_i^{|\gamma}p_i \mod J_{\alpha}.\]
  If $p_i^{|\gamma} \notin J+K$, then in fact $p_i^{|\gamma}q_i \in A_{\alpha}$. Hence, after rearranging the indices, we may write $b^{|\gamma} = \sum_{i=1}^\ell p_i^{|\gamma}q_i \mod A_{\alpha}$ where $p_i^{|\gamma} \in J+K$ for $i \leq \ell \leq m$. Hence, $\{b^{|\gamma} : \gamma \in \Res(b)\} $ is in the $K$-linear span of $\{q_1, \ldots , q_\ell\} \cup A_\alpha$, and since $\Res(b)$ is infinite, there must be a non-trivial $K$-linear dependence over $A_{\alpha}$.

  Pick some $\delta_1, \ldots , \delta_k \in K^{\times}$ and $\gamma_1, \ldots , \gamma_k \in \Res(b)$ arbitrarily close to $0$ such that $\sum_{i=1}^k \delta_ib^{|\gamma_i} \in A_{\alpha}$. Expanding the definition of $b$, we find that the tuple $\langle b_j^{|\gamma_i} \rangle_{1\leq j \leq n}^{1 \leq i \leq k}$ is linearly dependent over $A_\alpha$. Recall that $\deg_J(b_j^{|\gamma_i}) \leq \alpha$, and if the inequality is strict, then $b_j^{|\gamma_i} \in A_\alpha$. Thus the subtuple of the series of ordinal value $\omega^\alpha$ is also linearly dependent over $A_\alpha$. On the other hand, since $\QQ(b_1,\dots,b_n)$ holds, we also know that
  \[ \QQ(\langle b_j^{|\gamma_i} : \deg_J(b_j^{|\gamma_i}) = \alpha\rangle), \]
  as soon as $\gamma_1, \ldots , \gamma_k$ are sufficiently close to $0$, which contradicts $(*)_{\alpha}$.
\end{proof}

\subsection{The case \texorpdfstring{$\omega^{\alpha_1}+\omega^{\alpha_2}$}{ω\^{}α₁ + ω\^{}α₂}}
\begin{prop}\label{two terms}
  $(*)_{\omega^{\alpha_1}+\omega^{\alpha_2}}$ holds for all $\alpha_1 \geq \alpha_2$.
\end{prop}

We wish to proceed as in the proof of Proposition~\ref{suc}, but we now face new obstacles. Let us retrace the proof. Suppose $b = \sum_{i=1}^n \lambda_ib_i \in A_\alpha$ for some $\lambda_i$ not all zero. Write $b = \sum_{i=1}^m p_iq_i + r$ with $\deg_J(p_i) = \omega^{\alpha_1}$, $\deg_J(q_i) = \omega^{\alpha_2}$, $\deg_J(r) < \omega^{\alpha_1} + \omega^{\alpha_2}$. If we now consider $b^{|\gamma}$ for some arbitrary $\gamma \in \Res(b)$, we may have $\deg_J(r^{|\gamma}) \geq \deg_J^r(b) = \omega^{\alpha_1}$.

To overcome this issue, we find a large $\Gamma \subseteq \Res(b)$, namely $\ot(\Gamma) = v_J^p(b) = \omega^{\omega^{\alpha_2}}$ and $\sup(\Gamma) = 0$, on which we have $\deg_J(r^{|\gamma}) < \deg_J^r(b) = \omega^{\alpha_1}$.

\begin{lem}\label{bigger points are few}
  For every principal series $c$ and ordinal $\beta$,
  \[ \omega^{\beta} \cdot \ot(\{\gamma : \deg_J(c^{|\gamma}) \geq \beta\}) \leq v_J(c). \]
\end{lem}
\begin{proof}
  Let $c$, $\beta$ as in the assumptions and enumerate $\{\gamma : \deg_J(c^{|\gamma}) \geq \beta\}$ as $\{\gamma_i : i < \alpha\}$. By construction,
  \[ \ot(\supp(c) \cap [\gamma_i, \gamma_{i+1})) \geq \deg_J(c^{|\gamma_{i+1}}) \geq \omega^{\beta}. \]
  It follows by \cite[Lem.\ 4.7]{Ber2000} that
  \[ v_J(c) = \ot(c) \geq \ot(\supp(c) \cap [\gamma_0, 0)) \geq \omega^\beta \cdot \alpha. \qedhere\]
\end{proof}
It follows at once that for some $\delta < 0$ sufficiently small, $\BigP^{\omega^{\alpha_1}}(r) \cap [\delta, 0]$ has order type strictly smaller that $\omega^{\omega^{\alpha_2}}$, while $\ot(\Res(b)) = \omega^{\omega^{\alpha_2}}$ and $\sup(\Res(b)) = 0$ by \cite[Lem.\ 6.8]{Ber2000}. Therefore, $\Gamma = \Res(b) \setminus \BigP^{\omega^{\alpha_2}}(r)$ has order type $\omega^{\omega^{\alpha_2}}$ and $\sup(\Gamma) = 0$. By construction, for $\gamma \in \Gamma$ we have $\deg_J(r^{|\gamma}) < \omega^{\alpha_1}$.

We may now apply Proposition~\ref{leib} and deduce that for every $\gamma \in \Gamma$ close enough to $0$ we have
\[ b^{|\gamma} = \sum_{i=1}^m p_i^{|\gamma}q_i + q_i^{|\gamma}p_i \mod J_{\omega^{\alpha_1}}. \]
For every $1 \le i \le m$ we have $p_i^{|\gamma}, q_i^{|\gamma} \in J_{\omega^{\alpha_1}}$. Now $J_{\omega^{\alpha_1}} $ is an integral domain, thus the series $b^{|\gamma}$ for $\gamma \in \Gamma$ lie in a finitely generated $J_{\omega^{\alpha_1}}$-module. It follows that for some $k \leq 2m + 1$ there exist some $\delta_1, \ldots \delta_k \in J_{\omega^{\alpha_1}}$ not all zero and $\gamma_1, \ldots , \gamma_k \in \Gamma $ close enough to $0$ such that $\sum_{i=1}^k \delta_i b^{|\gamma_i} \in J_{\omega^{\alpha_1}} $. To conclude, we need to improve the coefficients $\delta_i$ so that they lie in $K$.

\begin{lem}\label{lem:small-princ-value}
  Let $c$, $d$, $e$ be series such that $\deg_J(c) < \deg_J(e)$ and $\deg_J(d) < \deg_J^p(e)$. Then $\deg_J(cd) < \deg_J(e)$.
\end{lem}
\begin{proof}
  By looking at the Cantor Normal form of $\deg_J(e)$, we observe that there is an ordinal $\epsilon > 0$ such that $\deg_J(c) \leq \deg_J^r(e) + \epsilon$ and $\epsilon < \deg_J^p(e)$. Since $\deg_J^p(e)$ is additively principal,
  \[ \deg_J(cd) \leq \deg_J^r(e) \oplus \epsilon \oplus \deg_J(d) < \deg_J^r(e) \oplus \deg_J^p(e) = \deg_J(e).\qedhere \]
\end{proof}

\begin{lem}\label{smallder}
  Let $c$, $d$ be series such that $\deg_J(d) < \deg_J^p(c)$. Then for every $\gamma$ close enough to $0$ we have $(cd)^{|\gamma} = cd^{|\gamma}$ mod $J_{\deg_J(c)}$.
\end{lem}
\begin{proof}
  For every $\crit_J(c) < \gamma < 0$ we have $\deg_J(c^{|\gamma}) < \deg_J(c)$, while for every $\crit_J(d) < \delta \leq 0$ we also have $\deg_J(d^{|\delta}) < \deg_J^p(c)$, thus by \prettyref{lem:small-princ-value} we have $\deg_J(d^{|\delta}c^{|\gamma}) < \deg_J(c)$. Hence, for every $\gamma$ sufficiently close to $0$ Proposition~\ref{conv} implies that
  \[(cd)^{|\gamma}=cd^{|\gamma} \mod J_{ \deg_J(c)}. \qedhere\]
\end{proof}
We now can show that one can find a linear combination with coefficients in $K$ rather than just in $J_{\omega^{\alpha_1}} $.
\begin{lem} \label{lincomb}
  Let $c_1,\ldots , c_\ell \notin J$ be series such that $\deg_J(c_i)=\deg_J(c_1)$ for every $ 1 \le i \le \ell$, and let $\varepsilon_1,\ldots, \varepsilon_\ell \in J_{\deg_J^p(c_1)}$. Then there exist some $\mu_1, \ldots ,  \mu_\ell \in K$ not all zero and $\gamma_0$ arbitrarily close to $0$ such that \[\left(\sum_{i=1}^\ell \varepsilon_ic_i\right)^{|\gamma_0} = \sum_{i=1}^\ell \mu_ic_i \mod J_{\deg_J(c_1)}.\]
\end{lem}
\begin{proof}
  Let $S = \bigcup_{i=1}^\ell \supp(\varepsilon_i)$, and let $\gamma_0 \in S$ be an isolated point arbitrarily close to $0$. Let $\mu_1, \ldots ,  \mu_\ell$ be the coefficients $\varepsilon_{1\gamma_0}, \ldots ,  \varepsilon_{\ell\gamma_0}$ respectively. Since $\gamma_0$ is isolated, we have $\varepsilon_i^{|\gamma_0} = \mu_i \mod J$; note in particular at least one such coefficient is not zero. Finally, by Lemma \ref{smallder} we have $(\sum_{i=1}^\ell \varepsilon_ic_i)^{|\gamma_0} = \sum_{i=1}^\ell \mu_ic_i \mod J_{\deg_J(c_1)}$, as required.
\end{proof}

Using Lemma~\ref{lincomb}, we obtain some $\gamma_0$ arbitrarily close to $0$ and $\mu_1 , \ldots , \mu_k \in K$ not all zero such that
\[ \left(\sum_{i=1}^k \delta_ib^{|\gamma_i}\right)^{|\gamma_0} = \sum_{i=1}^k \mu_ib^{|\gamma_i} \mod J_{\omega^{\alpha_1}}. \]
Note that moreover $b^{|\gamma_i} = \sum_{j=1}^n \lambda_j b_j^{|\gamma_i}$, where $\deg_J(b^{|\gamma_i})=\deg_J^r(b)$. It now remains to give a bound on $\deg_J(b_j^{|\gamma_i})$.
\begin{lem} \label{sameres}
  Let  $c_1, \ldots , c_\ell$ such that $\QQ(c_1, \ldots , c_\ell) $ holds. Suppose $k_1, \ldots , k_\ell \in K^{\times}$ and let $c = \sum_{i=1}^\ell k_ic_i $. Then for every $\gamma \in \Res(c)$ close enough to $0$ we have $\deg_J(c_i^{|\gamma}) \le \deg_J^r(c)$ for all $i = 1, \dots, \ell$.
\end{lem}
\begin{proof}
  Suppose by contradiction that there exists such $\gamma$ for which, after reordering the $b_i$'s, $\deg_J(b_1^{|\gamma}) = \ldots = \deg_J(b_s^{|\gamma}) > \deg_J^r(b)$ and $\deg_J(b_i^{|\gamma}) < \deg_J(b_1^{|\gamma})$ for $i > s$. Then $\deg_J(\sum_{i=1}^s k_ib_i^{|\gamma}) < \deg_J(b_1^{|\gamma})$, hence $\rv(b_1^{|\gamma}), \ldots , \rv(b_s^{|\gamma})$ are not linearly independent, which contradicts $\QQ(b_1^{|\gamma}, \ldots, b_s^{|\gamma})$.
\end{proof}
Since we may pick the $\gamma_i$'s to be arbitrarily close to $0$, we can ensure thanks to Lemma~\ref{sameres} that $\deg_J(b_j^{|\gamma_i})\le \deg_J^r(b)$ for every $1 \le i \le k $, $1 \le j \le n$. For each $j$, let $\gamma_{j,1}, \dots, \gamma_{j,m(j)}$ be an enumeration of the exponents $\gamma_i$ such that $v(b_j^{|\gamma_i}) = \deg_J^r(b)$. Observe that each $\gamma_i$ must appear at least once, as otherwise $\deg_J(b^{|\gamma_i}) < \deg_J^r(b)$, contradicting the fact that $\gamma_i$ in $\Res(b)$. In particular, $\sum_j m(j) \geq n \geq 1$.

However, we have $\QQ(b_1^{|\gamma_{1,1}}, \ldots, b_1^{|\gamma_{1,m(1)}}, \ldots, b_n^{|\gamma_{n,1}}, \ldots, b_n^{|\gamma_{n,m(n)}})$, thus the $b_j^{|\gamma_{i,k}}$'s cannot be $K$-linearly independent over $J_{\omega^{\alpha_1}}$, a contradiction. $\hfill \qed_{\text{\tiny Prop.~\ref{two terms}}}$

\bigskip

Combining with Proposition \ref{suc} and with Proposition \ref{genir} we obtain:
\begin{cor}\label{two terms plus n}
  Let $\alpha=\omega^{\alpha_1}+\omega^{\alpha_2}+k$ where $\alpha_1 \ge \alpha_2 $ and $k \in \mathbb{N}$.
  \begin{itemize}
    \item For every $b \in P_{\alpha}$ with $\QQ(b)$, we have that $b$, $\rv(b)$ are irreducible and cannot be represented as a sum of reducible elements in $\RV_J^{\alpha}$.
    \item $\rv(A_{\alpha}) $ has an infinite co-dimension as a vector space in $\RV_J^{\alpha}$.
    \item Let $b=\sum_{i=1}^m b_it^{\gamma_i}+r $ where $\sup(b)=0 $, $b_1, \ldots, b_m \in P_{\alpha}$, $\deg(r) < \alpha$. If $\QQ(b_1, \ldots, b_m)$, then $b$ is irreducible.
  \end{itemize}
\end{cor}
The above includes the conclusions of Theorems~\ref{printhm},~\ref{genthm} for the ordinals of the form $\alpha=\omega^{\alpha_1}+\omega^{\alpha_2}+k$.

\subsection{The case \texorpdfstring{$\omega^{\alpha_1}+\omega^{\alpha_2} +\omega^{\alpha_3}$}{ω\^{}α₁ + ω\^{}α₂ + ω\^{}α₃}}
We show now that we may deduce irreducibility by a different induction on the Cantor normal form of $\alpha$. Unlike the previous cases, in this induction we do not obtain $(*)_{\alpha}$, but irreducibility only.

As a result, we shall prove irreducibility for many series in $P_\alpha$ for $\alpha = 3\omega+1$ but not for $\alpha = 3\omega+2$. In fact, $3\omega + 2$ is now the smallest $\alpha$ for which we do not know whether there exists any irreducible element in $P_{\alpha}$.

We start by proving a crucial lemma for the inductive step:
\begin{lem} \label{strprin}
  Suppose $\deg_J^p(q)>\deg_J^p(p)>1 $. Then for every $\gamma$ close enough to $0$ we have $(pq)^{|\gamma} = p^{|\gamma}q \mod J_{\deg_J^r(pq)} $.
\end{lem}
\begin{proof}
  Suppose $\deg_J(pq) = \sum_{i=1}^n \omega^{\alpha_i}$ where $\alpha_i \ge \alpha_{i+1} $ for $1 \le i \le n-1 $. We have that
  \[ \deg_J(p) = \sum_{i \in A} \omega^{\alpha_i}, \quad \deg_J(q) = \sum_{i \in B} \omega^{\alpha_i} \]
  where $A \cup B = \{1, \dots, n\}$ and $A \cap B = \varnothing$. By assumption, we must have $n \in A$ and $\deg_J^p(p) = \omega^{\alpha_n} < \deg_J^p(q) = \omega^{\alpha_\ell}$, where $\ell < n$. We arrange the enumeration so that $\alpha_{\ell + 1} < \alpha_\ell$. Moreover, $\deg_J^r(pq) = \sum_{i=1}^{n-1} \omega^{\alpha_i}$.

  For every $\gamma$ close enough to $0$ we have
  \[ \deg_J(q^{|\gamma}) = \left(\sum_{i \in B \setminus \{\ell\}} \omega^{\alpha_i}\right) + \epsilon \]
  where $\epsilon < \omega^{\alpha_\ell}$. It follows that
  \[ \deg_J(q^{|\gamma}p) = \left(\sum_{i \neq \ell} \omega^{\alpha_i}\right) \oplus \varepsilon < \sum_{i = 1}^\ell \omega^{\alpha_i} \leq \deg_J^r(pq). \qedhere \]
\end{proof}

We now apply the above statement in two slightly different settings. First, we show an inductive step on the Cantor Normal Form of $\alpha$, under the assumption that the new term is strictly smaller than all the previous ones.

\begin{prop} \label{strict}
  Let $\alpha =\omega^{\alpha_1} + \ldots + \omega^{\alpha_n}$ be such that for every $b \in P_{\alpha}$ with $\QQ(b)$ we have that $b$, $\rv(b)$ are irreducible. Then for every $b$ principal with $\QQ(b)$ such that $\deg_J(b) = \alpha +  \omega^{\alpha_{n+1}}$ where $\alpha_n > \alpha_{n+1} > 0$ we have that $b$, $\rv(b)$ are irreducible.
\end{prop}
\begin{proof}
  Pick $b$ principal with $\QQ(b)$ and $\deg_J(b) = \alpha + \omega^{\alpha_{n+1}}$. Suppose by contradiction that $b=pq+r$ where $p$ and $q$ are not in $J + K$ and $\deg_J(r) < \deg_J(b)$. After possibly swapping $p$ and $q$, we may assume that $\deg_J^p(q) > \deg_J^p(p)$, and so that $\deg_J^p(p) = \omega^{\alpha_{n+1}} > 1$. Here $\deg_J^r(b) = \alpha$.

  Let $\Gamma = \Res(b) \setminus \BigP^\alpha(r)$. Note that by Lemma~\ref{bigger points are few}, $\Gamma$ has order type $\omega^{\alpha_{n+1}}$ and $\sup(\Gamma) = 0$. By Lemma~\ref{strprin}, for every $\gamma \in \Gamma$ close enough to $0$ we have $b^{|\gamma} = p^{|\gamma}q +r_{\gamma}$ where $\deg_J(r_{\gamma}) < \alpha$.

  If $\deg_J(q) < \alpha$, for any $\gamma$ as above, $p^{|\gamma}$ is not $J+K$, thus $\rv(b^{|\gamma})$ is not irreducible. By the assumption on $\alpha$, $\QQ(b^{|\gamma})$ does not hold, a contradiction against $\QQ(b)$.

  Therefore, $\deg_J(q) = \alpha$. It follows that for every $\gamma \in \Gamma$ close enough to $0$ we have $\deg_J(p^{|\gamma}) = 0$. Let $\gamma_1 \neq \gamma_2 \in \Gamma $ be two such exponents. Then $ p^{|\gamma_2}b^{|\gamma_1} - p^{|\gamma_1}b^{|\gamma_2} = p^{|\gamma_2}r_{\gamma_1} - p^{|\gamma_1}r_{\gamma_2}$. As $\deg_J(r_{\gamma_1}), \deg_J(r_{\gamma_2}) < \deg_J(q)$, we have $\deg_J(p^{|\gamma_2}r_{\gamma_1} - p^{|\gamma_1}r_{\gamma_2}) < \alpha$. Hence, $\rv_J(p^{|\gamma_2}b^{|\gamma_1})- \rv_J( p^{|\gamma_1}b^{|\gamma_2}) = 0 $, hence $\QQ(b^{|\gamma_1}, b^{|\gamma_2}) $ does not hold, another contradiction against $\QQ(b)$.
\end{proof}

We then prove a very similar conclusion when the Cantor Normal form has length three, but without the assumption that the last term is strictly smaller than the previous ones. We do this with a preliminary lemma.
\begin{prop}\label{princase}
  Suppose $b=pq+r$ where $\deg_J(q)=\deg_J^r(b)>\deg_J(p)> 0$ and $\deg_J(r)<\deg_J(b) $. Then there exist $\gamma_1 \neq \gamma_2 \in \Res(b) $ arbitrarily close to $0$ and $k_1, k_2 \in K $ not both $0$ such that $k_1b^{|\gamma_1}+k_2b^{|\gamma_2} \in A_{\deg_J(q)} $.
\end{prop}
\begin{proof}
  Let $\Gamma = (\Res(b) - \BigP^{\deg_J^r(b)}(r))$, which, as observed in the proof of Proposition~\ref{strict}, is infinite with supremum $0$. Note moreover that $\deg_J(p) = \deg_J^p(b)$, thus in particular $\deg_J(p) \leq \deg_J^p(q)$.

  If $\deg_J(p)<\deg_J^p(q)$, then as in the proof of Proposition \ref{strict} for every $\gamma_1 \neq \gamma_2 \in \Gamma $ close enough to $0$ there exist $k_1, k_2 \in K^{\times} $ such that $k_1b^{|\gamma_1}+k_2b^{|\gamma_2} \in J_{\deg_J(q)} \subseteq A_{\deg_J(q)} $.

  Otherwise, $\deg_J(p) = \deg_J^p(q)$. For every $\gamma \in \Gamma$ close enough to $0$, we have $b^{|\gamma} = p^{|\gamma}q + q^{|\gamma}p + r_{\gamma} $ where $\deg_J(r_{\gamma}) < \deg_J^r(b)$ and $\deg_J(p^{|\gamma}) < \deg_J(p) = \deg_J^p(q)$. In particular, $\deg_J(p^{|\gamma}r_\gamma) < \deg_J(q)$ by \prettyref{lem:small-princ-value}. Therefore, for every $\gamma_1 \neq  \gamma_2 \in \Gamma $ close enough to $0$ we have
  \[ p^{|\gamma_2}b^{|\gamma_1} - p^{|\gamma_1}b^{|\gamma_2} = (q^{|\gamma_1}p^{|\gamma_2} - q^{|\gamma_2}p^{|\gamma_1} )p \mod J_{\deg_J(q)}. \]

  As $\deg_J(p)=\deg_J^p(b) $, then by Lemma~\ref{lincomb} there exist $k_1,k_2 \in K $ not both $0$ and $\gamma $ arbitrarily close to $0$ such that
  \[ (p^{|\gamma_2}b^{|\gamma_1} - p^{|\gamma_1}b^{|\gamma_2})^{|\gamma} = k_1b^{|\gamma_1}+k_2b^{|\gamma_2} \mod J_{\deg_J(q)}. \]
  If $\deg_J(k_1b^{|\gamma_1} + k_2b^{|\gamma_2})<\deg_J(q) $ then we are done. Otherwise, we have
  \[\deg_J(((q^{|\gamma_1}p^{|\gamma_2} - q^{|\gamma_2}p^{|\gamma_1} )p)^{|\gamma})= \deg_J(q).\]
  When $\gamma$ is sufficiently small, for every $\varepsilon \in \cl(\supp(p))$ with $\gamma \leq \varepsilon < 0$ we have
  \[ \deg_J((q^{|\gamma_1}p^{|\gamma_2} - q^{|\gamma_2}p^{|\gamma_1})^{|\gamma - \varepsilon}) < \deg_J(q), \quad \deg_J(p^{|\epsilon}) < \deg_J(p) = \deg_J^p(q), \]
  with the former implied by \prettyref{lem:small-princ-value}. By a further application of \prettyref{lem:small-princ-value}, $\deg_J((q^{|\gamma_1}p^{|\gamma_2} - q^{|\gamma_2}p^{|\gamma_1} )^{|\gamma - \varepsilon}p^{|\varepsilon})<\deg_J(q)$. By Proposition~\ref{conv},
  \[((q^{|\gamma_1}p^{|\gamma_2} - q^{|\gamma_2}p^{|\gamma_1} )p)^{|\gamma} = (q^{|\gamma_1}p^{|\gamma_2} - q^{|\gamma_2}p^{|\gamma_1} )^{|\gamma}p \mod J_{\deg_J(q)}, \]
  which implies that
  \[ \deg_J((q^{|\gamma_1}p^{|\gamma_2} - q^{|\gamma_2}p^{|\gamma_1} )^{|\gamma}p)=\deg_J(q). \]
  As $0 < \deg_J(p) < \deg_J(q) $ we obtain that
  \[ (q^{|\gamma_1}p^{|\gamma_2} - q^{|\gamma_2}p^{|\gamma_1} )^{|\gamma}p \in A_{\deg_J(q)}\]
  and hence $k_1b^{|\gamma_1} +k_2b^{|\gamma_2} \in A_{\deg_J(q)} $ as required.
\end{proof}
\begin{cor}
  Let $\alpha = \omega^{\alpha_1}+\omega^{\alpha_2} + \omega^{\alpha_3}$ where $\alpha_1 \ge \alpha_2 \ge \alpha_3$. Let $b \in P_{\alpha}$ such that $\QQ(b)$. Then $b$, $\rv(b)$ are irreducible.
\end{cor}
\begin{proof}
  For $\alpha_2 > \alpha_3 > 0$ the conclusion follows from Propositions~\ref{two terms},~\ref{strict}; for $\alpha_3 = 0$, from Corollary~\ref{two terms plus n}. Therefore, we may assume that $\alpha_2 = \alpha_3 > 0$.

  Suppose by contradiction that $b = pq + r$ where $\deg_J(r) < \deg_J(b)$ and $0 < \deg_J(p) \leq \deg_J(q)$. Let $\Gamma = (\Res(b) - \BigP^{\deg_J^r(b)}(r))$. If $\deg_J^p(p) \neq \deg_J^p(q)$, then by Lemma \ref{strprin} for every $\gamma \in \Gamma$ close enough to $0$ we have $b^{|\gamma} = p^{|\gamma}q + r_{\gamma}$ or $b^{|\gamma} = pq^{|\gamma} + r_{\gamma}$, where $\deg_J(r_{\gamma}) < \deg_J^r(b) $, and we continue as in the proof of Proposition~\ref{strict}. Otherwise, $\alpha_2 = \alpha_3$ and $\deg_J(q) = \omega^{\alpha_1}+\omega^{\alpha_2} = \deg_J^r(b) $. By Lemma~\ref{princase} there exist $\gamma_1 \neq \gamma_2 \in \Res(b) $ close enough to $0$ and $k_1, k_2 \in K$ not both $0$ such that $k_1b^{|\gamma_1}+k_2b^{|\gamma_2} \in A_{\omega^{\alpha_1}+\omega^{\alpha_2}} $. As $\QQ(b^{|\gamma_1}, b^{|\gamma_2}) $ holds this contradicts $(*)_{\omega^{\alpha_1}+\omega^{\alpha_2}}$.
\end{proof}
Combining with Proposition~\ref{strict} we obtain:
\begin{cor}
  Let $\alpha = \sum_{i=1}^n \omega^{\alpha_i }$ where $\alpha_1 \ge \alpha_2 \ge \alpha_3$ and $\alpha_3 > \ldots > \alpha_n $. Let $b$ principal such that $\deg_J(b) = \alpha$ and $\QQ(b)$. Then $b$, $\rv(b)$ are irreducible.
\end{cor}
By Proposition~\ref{genir}, we obtain the following:
\begin{cor}
  Let $\alpha = \sum_{i=1}^n \omega^{\alpha_i }$ where $\alpha_1 \ge \alpha_2 \ge \alpha_3$ and $\alpha_3 > \ldots > \alpha_n $. Let $b=\sum_{i=1}^m b_it^{\gamma_i}+r $ where $\sup(b)=0 $, $b_1, \ldots , b_m \in P_{\alpha} $, $\deg(r)<\alpha $. If $\QQ(b_1, \ldots , b_m) $ then  $b$ is irreducible.
\end{cor}
This concludes the proofs of Theorems~\ref{printhm},~\ref{genthm}.

\bibliographystyle{alphaurl}
\bibliography{references}
\end{document}